\documentclass{amsart}

\usepackage{amsmath,amsthm,url}
\usepackage{mathrsfs}
\usepackage{latexsym}
\usepackage{amsthm,amsopn,amsfonts,amssymb,epsfig}

\usepackage{pstricks,graphicx}

\numberwithin{equation}{section}

\newtheorem{theorem}{Theorem}
\newtheorem{lemma}[theorem]{Lemma}
\newtheorem{proposition}[theorem]{Proposition}

\newtheorem{corollary}[theorem]{Corollary}

\theoremstyle{remark}

\begin{document}

\title[fixed trace $\beta$-Hermite ensembles]{Fixed trace $\beta$-Hermite ensembles: Asymptotic eigenvalue density and the
edge of the density }
\date{May, 2009\\
\ \ \ \ E-mail addresses: zhdasheng@gmail.com (D.-S. Zhou),
dzliumath@gmail.com \\ (D.-Z. Liu), fsttq@umac.mo (T. Qian).}
\maketitle

 \centerline{Da-Sheng
Zhou${}^{a}$, Dang-Zheng Liu${}^b$, Tao Qian${}^a$}
\bigskip
\centerline{{\it ${}^a$Department of Mathematics, University of
Macau, Av. Padre Tom\'{a}s Pereira,}}

\centerline{{\it Taipa, Macau,P.R. China}}

\centerline{{\it ${}^b$School of Mathematical Sciences, Peking
University, Beijing, 100871, P.R. China}}

\maketitle

\begin{abstract}
In the present paper, fixed trace $\beta$-Hermite ensembles
generalizing the fixed trace Gaussian Hermite ensemble are
considered. For all $\beta$, we prove the Wigner semicircle law for
these ensembles by using two different methods: one is the moment
equivalence method with the help of the matrix model for general
$\beta$, the other is to use asymptotic analysis tools. At the edge
of the density, we prove that the edge scaling limit for $\beta$-HE
implies the same limit for fixed trace  $\beta$-Hermite ensembles.
Consequently, explicit limit can be given for fixed trace GOE, GUE
and GSE. Furthermore, for even $\beta$, analogous to $\beta$-Hermite
ensembles, a multiple integral of the Konstevich type can be
obtained.

\end{abstract}
\noindent \textbf{Keywords}: Fixed trace ensembles; Asymptotic
analysis;
Semicircle law; Konstevich-type integral.\\
\noindent \textbf{2000 MSC}: 15A52, 41A60.\\

\section{Introduction and main results}

$\beta$-Hermite ensembles ($\beta$-HE) \cite{DE1,DF} generalize the
classical random matrix ensembles: Gaussian orthogonal, unitary and
symplectic ensembles (denoted by GOE, GUE and GSE for short, which
correspond to the Dyson index $\beta=1, 2$ and 4) from the
quantization index to the continuous exponents $\beta>0$. These
ensembles possess the joint probability density function (p.d.f.) of
real eigenvalues $x_{1},\ldots, x_{N}$ with the form
\begin{equation}
P_{\beta HE_{N}}(\textbf{x})=\frac{1}{Z_{\beta HE_{N}}}\prod_{1\leq
j<k\leq N}|x_{j}-x_{k}|^{\beta}\prod_{i=1}^{N}e^{-x_{i}^{2}/2}
\end{equation}where the normalization constant
$Z_{\beta HE_{N}}$  can be calculated with the help of the Selberg
integrals \cite{M}:
\begin{equation}
Z_{\beta
HE_{N}}=(2\pi)^{\frac{N}{2}}\prod_{j=1}^{N}\frac{\Gamma(1+\frac{\beta
j}{2})}{\Gamma(1+\frac{\beta}{2})}.
\end{equation}
Recently, Dumitriu and Eldeman \cite{DE1} have constructed a
tri-diagonal real symmetric matrices of the form
\begin{equation}
\label{matrixmodel} H_\beta \sim\frac{1}{\sqrt{2}}\begin{pmatrix}
  \mathrm{N}[0,2] & {\chi }_{(N-1)\beta}&  & &   \\
{\chi }_{(N-1)\beta} & \mathrm{N}[0,2]&{\chi }_{(N-2)\beta}& &  &    \\
  & {\chi }_{(N-2
  )\beta} & \mathrm{N}[0,2] & {\chi }_{(N-3)\beta} &   \\
 &\;\ddots &\;\ddots & \;\ddots  & &  \\
    & &{\chi }_{2\beta} &  \mathrm{N}[0,2] & {\chi }_{\beta} \\
 &  && {\chi }_{\beta} &\mathrm{N}[0,2]
\end{pmatrix}\end{equation}
where the $N$ diagonal and $N-1$ subdiagonal elements are mutually
independent, with standard normals on the diagonal, and $1/\sqrt{2}
\,\chi_{k\beta}$ on the subdiagonal. It is worth mentioning that the
p.d.f. of $1/\sqrt{2}\, \chi_{k\beta}$ is given by
$$\frac{2}{\Gamma{(k\beta/2)}}x^{k\beta-1}e^{-x^{2}}.$$
Furthermore, they proved that  the eigenvalues j.d.f. of $H_\beta $
was given by (1.1).

Basing on the p.d.f. of eigenvalues in Eq.\,(1.1), the (level)
density, or one-dimensional marginal eigenvalue density is defined
as follows:
\begin{align}
\label{density} \rho_{\beta HE_{N}}(x_{1})=\int_{\mathbb{R}^{N-1}}
\,P_{\beta HE_{N}}(\textbf{x})\,dx_{2}\cdots dx_{N}.\end{align} One
knows
 \cite{D,J} that the asymptotic eigenvalue density as $N
 \rightarrow\infty$ (density of states):
\begin{equation}
\label{semicirclelaw} \lim_{N\rightarrow\infty}\sqrt{2\beta
N}\large{\rho}_{\beta HE_{N}}(\sqrt{2\beta
N}x)=\rho_{\mathrm{W}}(x):=\begin{cases}
    \displaystyle\frac{2}{\pi}\sqrt{1-x^2}&  -1< x< 1, \\
    \displaystyle 0&
    |x|\geq 1.
    \end{cases}
\end{equation}
This result is  referred to as the Wigner semicircle law. In terms
of statistical physics, for any finite size $N$, we expect that most
of the eigenvalues concentrate in the interval
$(-\sqrt{2N},\sqrt{2N})$, referred to the ``bulk region" of
mechanical problem, while the scaled density decreases rapidly in
the vicinity of the spectrum edge $\approx\pm\sqrt{2N}$, referred to
the ``soft edge". At the edge of the spectrum, for Gaussian
orthogonal, unitary and symplectic ensembles, or $\beta=1, 2$ and 4,
a classical result \cite{F1,G1,G2} claims that the edge scaling
limit could be expressed in terms of Airy function. Explicitly, it
says: for $\beta=1,2,4$,
\begin{equation}
\label{edge124}
\lim_{N\rightarrow\infty}\frac{\sqrt{\beta}N^{5/6}}{\sqrt{2}}\rho_{\beta
HE_{N}}\left(\sqrt{2\beta
N}(1+\frac{x}{2N^{2/3}})\right)=\mathbf{Ai}_{\beta}(x)
\end{equation}
where
\begin{equation}
\label{need1} \mathbf{Ai}_{\beta}(x)=\begin{cases}
    \displaystyle (\textrm{Ai}^{\prime}(x))^2-x(\textrm{Ai}(x))^{2}+\frac{1}{2}\textrm{Ai}(x)\left(1-\int_{x}^{\infty}\textrm{Ai}(t)dt \right)  &  \beta=1, \\
    \displaystyle (\textrm{Ai}^{\prime}(x))^2-x(\textrm{Ai}(x))^{2}& \beta=2,\\
    \displaystyle (\textrm{Ai}^{\prime}(2 x))^2-2 x(\textrm{Ai}(2 x))^{2}-\textrm{Ai}(2 x)\int_{x}^{\infty}\textrm{Ai}(2 t)dt &
    \beta=4.
  \end{cases}
\end{equation}
Here, the Airy function of a real variable $x$ can be defined as
\begin{equation}
\textrm{Ai}(x)=\frac{1}{2\pi
i}\int_{-i\infty}^{i\infty}e^{v^{3}/3-xv}dv
\end{equation}
satisfying the equation
\begin{equation}
\textrm{Ai}^{\prime\prime}(x)=x\textrm{Ai}(x).
\end{equation}

Much finer correction terms to the large $N$ asymptotic expansions
of the eigenvalue density are considered in \cite{G1,G2}, and (1.5)
is also proved for $\beta=1,2,4$. Recently, for every even $\beta$,
Desrosiers and Forrester\cite{DF}, by applying the steepest descent
method, obtained a multiple integral of the Konstevich type at the
edges which constitutes a $\beta$-deformation of the Airy function.
That is,  for even $\beta$,
\begin{multline}
\label{generaledge}
\lim_{N\rightarrow\infty}\frac{\sqrt{\beta}N^{5/6}}{\sqrt{2}}\rho_{\beta HE_{N}}\left(\sqrt{2\beta N}(1+\frac{x}{2N^{2/3}})\right)\\
 =\frac{1}{2\pi}\left(\frac{4\pi}{\beta}\right)^{\beta/2}\frac{\Gamma(1+\beta/2)}{\prod_{j=2}^\beta
 {\Gamma(1+2/\beta)}^{-1}\Gamma(1+2j/\beta)}\,K_{\beta,\beta}(x)\end{multline}
where  $K_{\beta,\beta}(x)$ is a multiple integral of Konstevich
type or multiple Airy integrals defined by \cite{K2}:
\begin{equation}\label{Kbeta}K_{n,\beta}(x):=-
\frac{1}{(2\pi\mathrm{i})^{n}}\int_{-\mathrm{i}\infty}^{\mathrm{i}\infty}dv_1\cdots
\int_{-\mathrm{i}\infty}^{\mathrm{i}\infty}dv_n\prod_{j=1}^n
e^{v_j^3/3-xv_j}\prod_{1\leq k<l\leq
n}|v_k-v_l|^{4/\beta}\,.\end{equation}

Remark: When $\beta=2,4$, the right sides of (\ref{edge124}) and
(\ref{generaledge}) coincide. However, to stress the cases
$\beta=1,2,4$ and general $\beta$, we state them respectively.

\vspace{2mm}


In the present paper, we deal with fixed trace $\beta$-Hermite
ensembles and extend the properties (\ref{semicirclelaw}),
(\ref{edge124}) and (\ref{generaledge}) to these fixed trace
ensembles. First, let us give a review \cite{Rosenzweig,M}.
Proceeding from the analogy of a fixed energy in classical
statistical mechanics, Rosenzweig defines \cite{Rosenzweig} his
``fixed trace'' ensemble for a Gaussian real symmetric, Hermitian or
self-dual  matrix $H$ by the requirement that the trace of $H^{2}$
be fixed to a number $r^{2}$ with no other constraint. The number
$r$ is called the strength of the ensemble. The joint probability
density function for the matrix elements of $H$ is therefore given
by

\[P_{r}(H)=K_{r}^{-1}\,
\delta\big(\frac{1}{r^{2}}\text{tr}\,H^{2}-1\big)\]
 where $K_{r}$ is the normalization constant. Note that this probability density
 function is invariant under the conjugate action by orthogonal, unitary
 or symplectic groups, because of the invariance of the quantity
 $\text{tr}\,H^{2}$. Now Rosenzweig's  fixed trace
 ensemble has been extended to other ensembles, one of which is
 fixed trace  $\beta$-HE.
 With the help of $H_\beta$ in Eq.  (\ref{matrixmodel}), G. LeCa\"{e}r and R. Delannay \cite{LD1} define the associated fixed
trace $\beta$-HE as the ensemble of matrices:
$$F_{\beta}=\sqrt{N(N-1)/2}\,H_{\beta}\,/\sqrt{\text{tr}H_{\beta}^{2}}$$
satisfying $\text{tr}\,F^{2}_{\beta}=N(N-1)/2$.
 Its eigenvalue joint p.d.f has the form
\begin{equation}
P_{\beta FTE_{N}}(x_{1},x_{2},\cdots,x_{N})=\frac{1}{Z_{\beta
FTE_{N}}}\,\delta\,
\big(\sum_{i=1}^{N}x_{i}^{2}-\frac{N(N-1)}{2}\big) \prod_{1\leq j<k
\leq N} |x_{j}-x_{k}|^{\beta}
\end{equation}
where the normalization constant $Z_{\beta FTE_{N}}$ can be computed
by virtue of  variable substitution for the partition function
$Z_{\beta HE_{N}}$:
\begin{equation}
Z_{\beta
FTE_{N}}=(\frac{N(N-1)}{2})^{\frac{N_{\beta}-1}{2}}\frac{(2\pi)^{\frac{N}{2}}2^{-\frac{N_{\beta}}{2}+1}}{\Gamma(\frac{N_{\beta}}{2})}\prod_{j=1}^{N}\frac{\Gamma(1+\frac{j\beta}{2})}{\Gamma(1+\frac{\beta}{2})}
\end{equation}
where $N_{\beta}=N+\beta N(N-1)/2$. It is worth emphasizing that we
have chosen the square of the strength $r^{2}=N(N-1)/2$ since the
expectation of $\text{tr}H_{\beta}^{2}$ is $N(N-1)/2+N/\beta \approx
N(N-1)/2 $ as $N \rightarrow \infty$, Ref.\cite{M}.  Notice the
analogy: fixed trace $\beta$-HE bears the same relationship to
$\beta$-HE that the microcanonical ensembles to the canonical
ensemble in statistical physics \cite{Balian}. Besides, G .Akemann
et al \cite{AC} described further interesting physical features of
fixed trace ensembles due to the interaction among eigenvalues
introduced through a constraint.

As done usually for $\beta$-HE in Eq. (\ref{density}), the density
of fixed trace $\beta$-HE is written in the form
\begin{equation}
\label{densityforfixedtrace}
 \rho_{\beta
FTE_{N}}(x_{1})=\int_{\Omega_{N-1}} P_{\beta
FTE_{N}}(x_{1},x_{2},\cdots,x_{N}) d\,\sigma_{N-1}
\end{equation}
where $\Omega_{N-1}$ denotes the sphere $x^{2}_{2}+\cdots+
x^{2}_{N}=N(N-1)/2-x^{2}_{1}$, and $d\,\sigma_{N-1}$ denotes the
spherical measure.



The important thing to be noted about fixed trace GOE, GUE and GSE
is their moment equivalence with the associated Gaussian ensembles
of large dimensions (implying the semicircle law), see Mehta's book
\cite{M}, \textbf{Sect}.27.1, p.488. At the end of this section,
p.490, he writes:
\begin{quote}
It is not very clear whether this moment equivalence implies that
all local statistical properties of the eigenvalues in two sets of
ensembles are identical. This is so because these local properties
of eigenvalues may not be expressible only in terms of finite
moments of the matrix elements. \end{quote} In the Appendix of this
paper, Combining Rosenzweig's method \cite{M,Rosenzweig} and
Dumitriu and Edelman's matrix models (\ref{matrixmodel}), we present
the moment equivalence between fixed trace $\beta$-HE and $\beta$-HE
in the large $N$, which implies that the global density of
fixed-trace $\beta$-HE also fits the semi-circle law for all
$\beta$. We will derive the semicircle law using another quite
different method, and prove that the property of the spectrum edge
for $\beta$-HE implies the same property for fixed trace $\beta$-HE.
Recently, G$\ddot{o}$tze et al \cite{GG1,G3} have proven
universality of sine-kernel in the bulk for fixed GUE. In \cite{LZ},
asymptotic equivalence of local properties for correlation functions
at zero and the edge of the spectrum between fixed trace $\beta$-HE
and $\beta$-HE is proved, which implies universality of sine-kernel
at zero and airy-kernel at the edge for fixed trace GOE, GUE and
GSE. All these known results (to our knowledge), to some extent,
answer this open problem.

%

Now we can state our main results. Let $C_{c}(\mathbb{R})$ be the
set of all continuous functions on $\mathbb{R}$ with compact
support. For fixed trace $\beta$-HE, the scaled eigenvalue density
satisfies the Wiger semicircle law, i.e.,

\begin{theorem}
\label{theorem:global} Let $\rho_{\beta FTE_{N}}(x_{1})$ be the
eigenvalue density for fixed trace  $\beta$-HE, defined by
(\ref{densityforfixedtrace}). If $f(x)\in C_{c}(\mathbb{R})$, then
we have
\begin{equation}
\lim_{N\rightarrow \infty}\int_{\mathbb{R}}f(x)\sqrt{2N}\rho_{\beta
FTE_{N}}(\sqrt{2N}x)=\int_{\mathbb{R}}f(x)\rho_{\mathrm{W}}(x)dx
\nonumber
\end{equation}
where
\begin{equation*}
\rho_{\mathrm{W}}(x):=\begin{cases}
    \displaystyle\frac{2}{\pi}\sqrt{1-x^2}&  -1< x< 1, \\
    \displaystyle 0&
    |x|\geq 1.
    \end{cases}
 \end{equation*}
\end{theorem}

At the edge of the spectrum we can prove that that the property of
the spectrum edge for $\beta$-HE implies the same property for fixed
trace $\beta$-HE, thus we extend Desrosiers and Forrester's result
for even $\beta$ to fixed trace case.
\begin{theorem}
\label{theorem edge} Let $\rho_{\beta H E_{N}}(x_{1})$ and
$\rho_{\beta FTE_{N}}(x_{1})$ be the eigenvalue density of
$\beta$-HE and that of fixed trace, respectively,  defined by
(\ref{density}) and (\ref{densityforfixedtrace}). Assume that
$f(x)\in C_{c}(\mathbb{R})$. If $\forall\, h(t)\in
C_{c}(\mathbb{R})$,
\begin{align}
\label{conditon 1}
\lim_{N\rightarrow\infty}\int_{\mathbb{R}}h(t)\frac{\sqrt{\beta}N^{5/6}}{\sqrt{2}}\rho_{\beta
HE_{N}}\left(\sqrt{2 \beta N}(1+\frac{t}{2N^{2/3}})\right)\,dt
\end{align}
exists, then
\begin{multline}
\lim_{N\rightarrow\infty}\int_{\mathbb{R}}f(x)\frac{N^{5/6}}{\sqrt{2}}\rho_{\beta
FTE_{N}}\left(\sqrt{2N}(1+\frac{x}{2N^{2/3}})\right)dx\\ =
\lim_{N\rightarrow\infty}\int_{\mathbb{R}}f(t)\frac{\sqrt{\beta}N^{5/6}}{\sqrt{2}}\rho_{\beta
HE_{N}}\left(\sqrt{2 \beta N}(1+\frac{t}{2N^{2/3}})\right)\,dt.
\end{multline}
\end{theorem}

It immediately follows from (\ref{edge124}), (\ref{generaledge}) and
Theorem \ref{theorem edge} that
\begin{corollary}
Let $\rho_{\beta FTE_{N}}(x_{1})$ be the eigenvalue density for
fixed trace  $\beta$-HE, defined by (\ref{densityforfixedtrace}). If
$f(x)\in C_{c}(\mathbb{R})$, then at the edge of the spectrum one
has
\begin{equation}
\lim_{N\rightarrow\infty}\int_{\mathbb{R}}f(x)\frac{N^{5/6}}{\sqrt{2}}\rho_{\beta
FTE_{N}}\left(\sqrt{2N}(1+\frac{x}{2N^{2/3}})\right)dx=\int_{\mathbb{R}}f(t)\mathbf{Ai}_{\beta}(t)dt
\end{equation} for $\beta=1,2,4$
 and
\begin{multline*}\lim_{N\rightarrow\infty}\int_{\mathbb{R}}f(x)\frac{N^{5/6}}{\sqrt{2}}\rho_{\beta FTE_{N}}\left(\sqrt{2N}(1+\frac{x}{2N^{2/3}})\right)dx\\
 =\frac{1}{2\pi}\left(\frac{4\pi}{\beta}\right)^{\beta/2}\frac{\Gamma(1+\beta/2)}{\prod_{j=2}^\beta
 {\Gamma(1+2/\beta)}^{-1}\Gamma(1+2j/\beta)}\,\int_{\mathbb{R}}f(t)K_{\beta,\beta}(t)dt\end{multline*}
for even $\beta$. Here $\mathbf{Ai}_{\beta}(x)$ and
$K_{\beta,\beta}(x)$ are defined by (\ref{need1}) and (\ref{Kbeta})
respectively.

\end{corollary}

Theorem 1 will be proved in Sect. 3 and Theorem 2 in Sect. 4 after
the preparatory work in Sect. 2.

\section{An upper bound for the level density}

 \label{3} In this section, we will
give an estimation of the level density for fixed trace $\beta$-HE,
with the help of the maximum of Vandermonde determinant on the
sphere by Stieltjes \cite{S}. For the convenience of our argument,
let's first state Stieltjes's remarkable result as a lemma.
\begin{lemma}
Let us consider a unit mass at each of the variable points
$x_{1},x_{2},\cdots,x_{N}$ in the interval $[-\infty,+\infty]$ such
that
\begin{equation}
\sum_{i=1}^{N}x_{i}^{2}\leq \frac{N(N-1)}{2}\nonumber,
\end{equation}
then the maximal of
\begin{equation}
V(x_{1},x_{2},\cdots,x_{N})=\prod_{1\leq j<k \leq
N}|x_{j}-x_{k}|^{2}\nonumber
\end{equation}
is attained if and only if the ${x_{j}}$ are the zeros of the
Hermite polynomial $H_{N}(x)$, and the maximal is
\begin{equation}
\max_{\sum_{i=1}^{N}x_{i}^{2}\leq \frac{N(N-1)}{2}}
V(x_{1},x_{2},\cdots,x_{N})=2^{-\frac{N(N-1)}{2}}\prod_{v=1}^{N}e^{v\ln
v}.
\end{equation}
\end{lemma}
Note that Stieltjes's result can be interpreted as a electrostatic
problem and the maximum position corresponds to the condition of
electrostatic equilibrium.

The following proposition gives an estimation of the level density,
which gives a global control of the density for fixed trace
$\beta$-HE.
\begin{proposition}
\label{proposition 1} For the level density $\rho_{\beta
FTE_{N}}(x_{1})$, rescaling
$$x_{1}=\sqrt{N(N-1)/2}\, x,\ -1\leq x\leq 1,$$ then we have
\begin{equation} \rho_{\beta
FTE_{N}}\Big(\sqrt{\frac{N(N-1)}{2}}\,x\Big)\leq
e^{W_{N\beta}N}(1-x^{2})^{\frac{N-2}{2}}
\end{equation}
where $W_{N\beta}=\ln C_{\beta}+o(1)$ and
\begin{equation}
C_{\beta}=\exp\big(1-\ln\sqrt{2\pi}+\frac{\beta}{2}-\frac{\beta}{2}\ln(\frac{\beta}{2})\big)\,
\Gamma(1+\frac{\beta}{2}).
\end{equation}
\end{proposition}
\begin{proof}
From $(1.14)$ and $(2.1)$, we see that
\begin{equation}
\rho_{\beta FTE_{N}}(x_{1})\leq \frac{1}{Z_{\beta
FTE_{N}}}2^{-\frac{\beta N(N-1)}{4}}\big(\prod_{v=1}^{N}e^{v\ln
v}\big)^{\frac{\beta}{2}}\int_{x_{2}^{2}+\cdots+x_{N}^{2}=\frac{N(N-1)}{2}-x_{1}^{2}}d\,\sigma_{N-1},
\end{equation}
where $\sigma_{N-1}$ denotes $N-2$ dimensional spherical measure. By
the formula for surface area of the sphere, a direct calculation
shows that
\begin{align}
&\rho_{\beta FTE_{N}}(x_{1})\leq \frac{1}{Z_{\beta
FTE_{N}}}2^{-\frac{\beta N(N-1)}{4}}\big(\prod_{v=1}^{N}e^{v\ln
v}\big)^{\frac{\beta}{2}}\frac{2 \pi
^{\frac{N-1}{2}}}{\Gamma(\frac{N-1}{2})}\big(\frac{N(N-1)}{2}-\frac{N(N-1)}{2}x^{2}\big)^{\frac{N-2}{2}}\nonumber\\
&=\frac{2^{-\frac{\beta N(N-1)}{4}}\big(\prod_{v=1}^{N}e^{v\ln
v}\big)^{\frac{\beta}{2}}2\pi
^{\frac{N-1}{2}}\big(\frac{N(N-1)}{2}-\frac{N(N-1)}{2}x^{2}\big)^{\frac{N-2}{2}}\Gamma(\frac{N_{\beta}}{2})}{\Gamma(\frac{N-1}{2})(\frac{N(N-1)}{2})^{\frac{N_{\beta}-1}{2}}(2\pi)^{\frac{N}{2}}2^{-\frac{N_{\beta}}{2}+1}
}\prod_{j=1}^{N}\frac{\Gamma(1+\frac{\beta}{2})}{\Gamma(1+\frac{j\beta}{2})}\nonumber\\
&=\frac{(1-x^{2})^{\frac{N-2}{2}}\,e^{\frac{\beta}{2}\sum_{v=1}^{N}v\ln
v}}
{\sqrt{\pi}{\Gamma(\frac{N-1}{2})}}\big(\frac{N(N-1)}{2}\big)^{\frac{N-2}{2}}\Big(\frac{\Gamma(\frac{N_{\beta}}{2})}{(\frac{N(N-1)}{2})^{\frac{N_{\beta}-1}{2}}}\Big)
\prod_{j=1}^{N}\frac{\Gamma(1+\frac{\beta}{2})}{\Gamma(1+\frac{j\beta}{2})}\\
&\triangleq (1-x^{2})^{\frac{N-2}{2}} g_{N\beta}.\end{align} It is
sufficient to consider $\sqrt[N]{g_{N\beta}}$ as $N\rightarrow
\infty$. Using
 Stirling's formula  for the gamma function,
\begin{equation}
\label{gamma}
\Gamma(x)=(2\pi)^{1/2}e^{-x}x^{x-\frac{1}{2}}(1+O(\frac{1}{x}))
\end{equation}
for the large $x$. For the large $N$,
\begin{align}
\Gamma(\frac{N-1}{2})&=(2\pi)^{1/2}e^{-\frac{N-1}{2}}\,(\frac{N-1}{2})^{\frac{N}{2}-1}(1+O(\frac{1}{N})),\\
\Gamma(\frac{N_{\beta}}{2})&=(2\pi)^{1/2}e^{-N_{\beta}/2}\,(\frac{N_{\beta}}{2})^{\frac{N_{\beta}-1}{2}}(1+O(\frac{1}{N})).
\end{align}
Note that
$$\Big(\frac{N_{\beta}}{N(N-1)}\Big)^{\frac{N_{\beta}-1}{2}}=\big(\frac{\beta}{2}\big)^{\frac{N_{\beta}-1}{2}}\,
\Big(1+\frac{2}{\beta (N-1)}\Big)^{\frac{N_{\beta}-1}{2}},$$
 thus $g_{N\beta}$ can be rewritten as
\begin{equation}
 g_{N\beta}=\frac{1}{\sqrt{\pi}}\, e^{\frac{N-1}{2}} \big(\Gamma(1+\frac{\beta}{2})\big)^{N}
\Big(1+\frac{2}{\beta
 (N-1)}\Big)^{\frac{N_{\beta}-1}{2}}(1+O(\frac{1}{N}))\,\tilde{g}_{N\beta}
\end{equation}
where
$$\tilde{g}_{N\beta}=\frac{\exp(\frac{\beta}{2}\sum_{v=1}^{N}v\ln v)}{\prod_{j=1}^{N}\Gamma(1+\frac{j\beta}{2})
}\, N^{\frac{N-2}{2}} e^{-\frac{N_{\beta}}{2}}
\,\big(\frac{\beta}{2}\big)^{\frac{N_{\beta}-1}{2}}.
$$
Observe that
\begin{align}
\lim_{N\rightarrow \infty}\left(\frac{1}{\sqrt{\pi}}\,
e^{\frac{N-1}{2}} \big(\Gamma(1+\frac{\beta}{2})\big)^{N}
\Big(1+\frac{2}{\beta
 (N-1)}\Big)^{\frac{N_{\beta}-1}{2}}(1+O(\frac{1}{N}))\right)^{1/N}
&=e\,\Gamma(1+\frac{\beta}{2}).
\end{align}
On the other hand, using Stolz's rule,
\begin{align}
&\lim_{N\rightarrow\infty}\frac{1}{N}\ln \tilde{g}_{N\beta}\nonumber\\
&=\lim_{N\rightarrow\infty}\frac{1}{N}\left(\frac{\beta}{2}\sum_{v=1}^{N}v\ln
v+\frac{N-2}{2}\ln N
-\frac{N_{\beta}}{2}-\sum_{j=1}^{N}\ln\Gamma(1+\frac{j\beta}{2})+\frac{N_{\beta}-1}{2}\ln(\frac{\beta}{2})\right)\nonumber\\
&=\lim_{N\rightarrow\infty}\left(\frac{\beta}{2}N\ln
N+\frac{1}{2}\ln N-\frac{N-3}{2}\ln (1-\frac{1}{N})
-\ln\Gamma(1+\frac{N\beta}{2})+\frac{1}{2}(\ln\frac{\beta}{2}-1)(1-\beta+\beta
N)\right)\nonumber\\
&=-\ln\sqrt{2\pi}+\frac{\beta}{2}-\frac{\beta}{2}\ln(\frac{\beta}{2}).
\end{align}
In the above calculation, we make use of the following asymptotic
expansion:
\begin{equation}
\ln\Gamma(1+\frac{N}{2}\beta)=\ln\sqrt{2\pi}-\frac{N}{2}\beta+(\frac{N}{2}\beta+\frac{1}{2})\ln(\frac{N}{2}\beta)
+O(\frac{1}{N}).
\end{equation}

Combining (2.11), (2.12) and (2.6), this completes the proof of
Proposition 5. 
\end{proof}

\section{Proof of Theorem \ref{theorem:global}}
To prove Theorem \ref{theorem:global}, we will first prove the
following Lemma \ref{lemma:Globa almost equivalent}, which means
that the level density of $\beta$-HE defined by (\ref{density}) and
that of fixed trace $\beta$-HE given by (\ref{densityforfixedtrace})
are \emph{almost equivalent}. Our arguments depend on the following
integral equation
 \begin{equation}
 \label{integral equation}
  \rho_{\beta
 HE_{N}}(x_{1})=\frac{1}{C_{N\beta}}\int_{|x_{1}|}^{+\infty}e^{-r^{2}/2}r^{N_{\beta}-2}\rho_{\beta
 FTE_{N},1}(\frac{x_{1}}{r})dr,
 \end{equation}
obtained in \cite{LD1,DL} where
$C_{N\beta}=\Gamma(N_{\beta}/2)2^{N_{\beta}/2-1}$. Here $\rho_{\beta
FTE_{N},1}(x_{1})$ denoting the level density of fixed trace
$\beta$-HE whose strength is 1, is defined by
\begin{equation}
\label{definition: the level density: strength 1} \rho_{\beta
FTE_{N},1}(x_{1})=\frac{1}{Z_{\beta FTE_{N},1}}
\int_{\mathbb{R}^{N-1}} \delta (\sum_{i=1}^{N}x_{i}^{2}-1)
\prod_{1\leq j<k \leq N} |x_{j}-x_{k}|^{\beta}dx_{2}\ldots d x_{N}
\end{equation}
where the partition function 
\begin{equation}
Z_{\beta
FTE_{N},1}=\frac{(2\pi)^{\frac{N}{2}}2^{-\frac{N_{\beta}}{2}+1}}{\Gamma(\frac{N_{\beta}}{2})}\prod_{j=1}^{N}\frac{\Gamma(1+\frac{j\beta}{2})}{\Gamma(1+\frac{\beta}{2})}.
\end{equation}
A direct calculation shows
\begin{equation}
\label{relation} \rho_{\beta
FTE_{N},1}(x)=\sqrt{\frac{N(N-1)}{2}}\rho_{\beta
FTE_{N}}(\sqrt{\frac{N(N-1)}{2}}x).
\end{equation}
It follows from with Proposition \ref{proposition 1} that
\begin{equation}
\label{level density estimate: strength 1} \rho_{\beta
FTE_{N},1}(x)\leq
\sqrt{\frac{N(N-1)}{2}}e^{W_{N\beta}N}(1-x^{2})^{\frac{N-2}{2}}
\end{equation}
for any $-1\leq x\leq 1$.  Now we are ready to state the following
\emph{almost equivalent} lemma about the two ensembles for all
$\beta>0$, which has been obtained in \cite{xuda} for $\beta=2$
without rigorous arguments.
\begin{lemma}
\label{lemma:Globa almost equivalent} Let $-1\leq x\leq1$ be fixed.
For the level density of fixed trace $\beta$-HE defined by
(\ref{density}) and that of $\beta$-HE given by
(\ref{densityforfixedtrace}), we have
\begin{equation}
\rho_{\beta
HE_{N}}(\sqrt{2N\beta}x)=\big(\frac{1}{\sqrt{\beta}}+O(\frac{1}{N})\big)\,\rho_{\beta
FTE_{N}}\big(\sqrt{2N}x(1+O(\alpha_{N}))\big)+O(e^{-\beta
N^{2(1-\theta)}(1+o(1))})
\end{equation}
for large $N$ where $\alpha_{N} =\frac{1}{N^{\theta}}$, $0<\theta<
0.5$.
\end{lemma}
\begin{proof}
Dividing the right hand side of the integral equation
$(\ref{integral equation})$ into three parts, then
\begin{align}
\label{divide:global}
 \rho&_{\beta HE_{N}}(x_{1})\nonumber\\
&=\frac{1}{C_{N\beta}}(\int_{|x_{1}|}^{\sqrt{\beta}N(\frac{1}{\sqrt{2}}-\alpha_{N})}+\int_{\sqrt{\beta}N(\frac{1}{\sqrt{2}}+\alpha_{N})}^{+\infty}
+\int_{\sqrt{\beta}N(\frac{1}{\sqrt{2}}-\alpha_{N})}^{\sqrt{\beta}N(\frac{1}{\sqrt{2}}+\alpha_{N})})\,
e^{-r^{2}/2}r^{N_{\beta}-2}\,\rho_{\beta FTE_{N},1}(\frac{x_{1}}{r})dr\nonumber\\
&=I+II+III.
\end{align}
Next we will estimate $I$ and $II$ respectively. For large $N$, the
function $e^{-r^{2}/2}r^{N_{\beta}-2}$ attains its maximum at
$\sqrt{N_{\beta}-2}$ , which satisfying
\begin{equation}
\sqrt{N_{\beta}-2}>\sqrt{\beta}N(\frac{1}{\sqrt{2}}-\alpha_{N}).
\end{equation}
Together with $(\ref{gamma})$ and $(\ref{level density estimate:
strength 1})$, $I$ can be dominated by
\begin{align}
I&\leq\frac{1}{C_{N\beta}}\int_{|x_{1}|}^{\sqrt{\beta}N(\frac{1}{\sqrt{2}}-\alpha_{N})}e^{-r^{2}/2}r^{N_{\beta}-2}dr e^{W_{N\beta}N}\sqrt{\frac{N(N-1)}{2}}\nonumber\\
&\leq\frac{e^{-(\sqrt{\beta}N(\frac{1}{\sqrt{2}}-\alpha_{N}))^{2}/2}(\sqrt{\beta}N)^{N_{\beta}-2}(\frac{1}{\sqrt{2}}-\alpha_{N})^{N_{\beta}-2}}{\Gamma{(\frac{N_{\beta}}{2})}2^{\frac{N_{\beta}}{2}-1}}(\sqrt{\beta}N(\frac{1}{\sqrt{2}}-\alpha_{N})-|x_{1}|)e^{W_{N\beta}N}\sqrt{\frac{N(N-1)}{2}}\nonumber\\
&\leq\frac{e^{-(\sqrt{\beta}N(\frac{1}{\sqrt{2}}-\alpha_{N}))^{2}/2}(\sqrt{\beta}N)^{N_{\beta}-2}(\frac{1}{\sqrt{2}}-\alpha_{N})^{N_{\beta}-2}}{\sqrt{2\pi}e^{-\frac{N_{\beta}}{2}}(N_{\beta}/2)^{\frac{N_{\beta}-1}{2}}(1+O(1/N^{2}))2^{\frac{N_{\beta}}{2}-1}}(\sqrt{\beta}N(\frac{1}{\sqrt{2}}-\alpha_{N})-|x_{1}|)e^{W_{N\beta}N}\sqrt{\frac{N(N-1)}{2}}\nonumber\\
&\leq
C^{\prime}\frac{e^{-(\sqrt{\beta}N(\frac{1}{\sqrt{2}}-\alpha_{N}))^{2}/2}(\sqrt{\beta}N)^{N_{\beta}}(\frac{1}{\sqrt{2}}-\alpha_{N})^{N_{\beta}}}{e^{-N_{\beta}/2}N_{\beta}^{\frac{N_{\beta}}{2}}}N
e^{W_{N\beta}N}\label{3.9}
\end{align}
where we have used the fact that
\begin{equation}
\sqrt{\frac{N(N-1)}{2}}(\frac{1}{\sqrt{2}}-\alpha_{N})^{-2}(\sqrt{\beta}N)^{-2}\Big(\sqrt{\beta}N(\frac{1}{\sqrt{2}}-\alpha_{N})-|x_{1}|\Big)\,\frac{1}{2\sqrt{\pi}(N_{\beta})^{-1/2}}\leq
C^{\prime}N
\end{equation}
for large $N$. Here $C^{\prime}$ is a constant only depending on
$\beta$. We will estimate the right hand side of (\ref{3.9}).
Expanding
$$\ln (
N_{\beta}^{N_{\beta}/2})=\frac{N_{\beta}}{2}\big(\ln(\frac{\beta}{2}N^{2})+\ln(1+(\frac{2}{\beta}-1)\frac{1}{N})\big)$$
 and by a direct calculation, one obtains
\begin{equation}
\frac{(\sqrt{\beta}N)^{N_{\beta}}}{N_{\beta}^{N_{\beta}/2}}=\exp(\frac{N_{\beta}}{2}\ln
2-\frac{N}{2}+\frac{\beta N}{4}+O(1)).
\end{equation}
Again, expanding
$$\ln (\frac{1}{\sqrt{2}}-\alpha_{N})=-\ln \sqrt{2}-\sqrt{2}\alpha_{N}-\alpha_{N}^{2}+O(\alpha_{N}^{3}).$$
therefore $I$ can be dominated by
\begin{multline*}
I\leq
C^{\prime}\exp\Big(-\frac{\beta}{2}(\frac{1}{\sqrt{2}}-\alpha_{N})^{2}N^{2}-
\big(\ln
\sqrt{2}+\sqrt{2}\alpha_{N}+\alpha_{N}^{2}+O(\alpha_{N}^{3})\big)N_{\beta}+\frac{1}{2}N_{\beta}+\frac{\ln
2}{2}N_{\beta}\\-\frac{N}{2}+\frac{\beta N}{4}+O(1) +W_{N\beta}N+\ln
N\Big)\end{multline*}\begin{align} \label{I result}
&=C^{\prime}\exp\Big(-\frac{\beta}{2}(\frac{1}{\sqrt{2}}-\alpha_{N})^{2}N^{2}-
\frac{\beta}{2}\big(\sqrt{2}\alpha_{N}+\alpha_{N}^{2}+O(\alpha_{N}^{3})\big)N^{2}+\frac{\beta}{4}N^{2}+W_{N\beta}N+\ln N+O(1)\Big)\nonumber\\
&=C^{\prime}\exp\big(-\beta\alpha_{N}^{2}N^{2}+\beta
N^{2}O(\alpha_{N}^{3})+W_{N\beta}N+\ln N+O(1)\big)\nonumber\\
&=C^{\prime}\exp\big(-\beta
N^{2-2\theta}(1+O(N^{-\theta})+O(N^{2\theta-2})+O(N^{2\theta-1}))\big)=O(e^{-\beta
N^{2(1-\theta)}(1+o(1))}).
\end{align}
Here  we should take $\theta\in (0,0.5)$. For the convenience of our
argument, write
$$O\Big(\exp\big(-\beta
N^{2-2\theta}(1+O(N^{-\theta})+O(N^{2-2\theta})+O(N^{2\theta-1}))\big)\Big)\triangleq\Xi_{N}.$$
Similarly , we can estimate $II$. For large $N$, the function
$e^{-r^{2}/2}r^{N_{\beta}}$ attains its maximum at
$\sqrt{N_{\beta}}$ with the condition
\begin{equation}
\sqrt{N_{\beta}}<\sqrt{\beta}N(\frac{1}{\sqrt{2}}+\alpha_{N}).
\end{equation}
This shows that $II$ can be dominated by
\begin{align}
\label{II estimate}
II&<\frac{1}{C_{N\beta}}\int_{\sqrt{\beta}N(\frac{1}{\sqrt{2}}+\alpha_{N})}^{+\infty}e^{-r^{2}/2}r^{N_{\beta}}r^{-2}dre^{W_{N\beta}N}\sqrt{\frac{N(N-1)}{2}}\nonumber\\
&\leq\frac{e^{-(\sqrt{\beta}N(\frac{1}{\sqrt{2}}+\alpha_{N}))^{2}/2}(\sqrt{\beta}N)^{N_{\beta}}(\frac{1}{\sqrt{2}}+\alpha_{N})^{N_{\beta}}}{\Gamma{(\frac{N_{\beta}}{2})}2^{\frac{N_{\beta}}{2}-1}}
\int_{\sqrt{\beta}N(\frac{1}{\sqrt{2}}+\alpha_{N})}^{+\infty}r^{-2}dr e^{W_{N\beta}N}\sqrt{\frac{N(N-1)}{2}}\nonumber\\
&\leq C^{\prime\prime}\exp\big(-\beta
N^{2-2\theta}(1+O(N^{2\theta-2})+O(N^{-\theta})+O(N^{2\theta-1}))\big)=\Xi_{N}.
\end{align}
Here $C^{\prime\prime}$ is a constant only depending on $\beta$. By
(\ref{I result}) and (\ref{II estimate}), the identity
(\ref{divide:global}) can be reduced to
\begin{equation}
\rho_{\beta
HE_{N}}(x_{1})=\frac{1}{C_{N\beta}}\int_{\sqrt{\beta}N(\frac{1}{\sqrt{2}}-\alpha_{N})}^{\sqrt{\beta}N(\frac{1}{\sqrt{2}}+\alpha_{N})}e^{-r^{2}/2}r^{N_{\beta}-2}\rho_{\beta
FTE_{N},1}(\frac{x_{1}}{r})dr+\Xi_{N}.
\end{equation}
Using the intermediate value theorem of integral,
\begin{equation}
\label{zhong zhi ding li hou } \rho_{\beta
HE_{N}}(x_{1})=\rho_{\beta FTE_{N},1}(\frac{x_{1}}{\xi_{N}(x_{1})})
\frac{1}{C_{N\beta}}\int_{\sqrt{\beta}N(\frac{1}{\sqrt{2}}-\alpha_{N})}^{\sqrt{\beta}N
(\frac{1}{\sqrt{2}}+\alpha_{N})}e^{-r^{2}/2}r^{N_{\beta}-2}dr+\Xi_{N}
\end{equation}
where $ \sqrt{\beta}N(1/\sqrt{2}-\alpha_{N})\leq\xi_{N}(x_{1})\leq
\sqrt{\beta}N(1/\sqrt{2}+\alpha_{N})$. If we repeat the procedure of
obtaining the estimate of $I$ and $II$, then
\begin{align}
\frac{1}{C_{N\beta}}\int_{0}^{\sqrt{\beta}N(\frac{1}{\sqrt{2}}-\alpha_{N})}e^{-\frac{r^{2}}{2}}r^{N_{\beta}-2}dr=O(e^{-\beta
N^{2(1-\theta)}(1+o(1))}),\label{1}\\
\frac{1}{C_{N\beta}}\int_{\sqrt{\beta}N(\frac{1}{\sqrt{2}}+\alpha_{N})}^{+\infty}e^{-\frac{r^{2}}{2}}r^{N_{\beta}-2}dr=O(e^{-\beta
N^{2(1-\theta)}(1+o(1))})\label{2}
\end{align}
for $0<\theta<1$.  Actually, the process of the argument of
(\ref{3.9})
 tells us that the term
$\sqrt{(N(N-1))/2}e^{W_{N\beta}}$ will disappear in that of
obtaining (3.17). Hence the left hand side of (\ref{1}) can be
dominated by
\begin{align}
\label{3.19}C_{1}\frac{e^{-(\sqrt{\beta}N(\frac{1}{\sqrt{2}}-\alpha_{N}))^{2}/2}(\sqrt{\beta}N)^{N_{\beta}}(\frac{1}{\sqrt{2}}-\alpha_{N})^{N_{\beta}}}{e^{-N_{\beta}/2}N_{\beta}^{\frac{N_{\beta}}{2}}}
\end{align}
where
\begin{equation}
C_{1}=(\frac{1}{\sqrt{2}}-\alpha_{N})^{-2}(\sqrt{\beta}N)^{-2}\Big(\sqrt{\beta}N(\frac{1}{\sqrt{2}}-\alpha_{N})\Big)\,\frac{1}{2\sqrt{\pi}(N_{\beta})^{-1/2}}.
\end{equation}
Here it is easy to see that $C_{1}$ is bounded for large $N$.
Repeated arguments similar with (\ref{I result}) show that
(\ref{3.19}) can be dominated by $\exp\big(-\beta
N^{2-2\theta}(1+O(N^{-\theta})+O(N^{2\theta-2}))\big)$. The same
method can be used to obtain (\ref{2}). Hence, for $0<\theta<1$,
\begin{align}
\label{media integral}
&\frac{1}{C_{N\beta}}\int_{\sqrt{\beta}N(\frac{1}{\sqrt{2}}-\alpha_{N})}^{\sqrt{\beta}N(\frac{1}{\sqrt{2}}+\alpha_{N})}e^{-\frac{r^{2}}{2}}r^{N_{\beta}-2}dr
=\frac{1}{C_{N\beta}}\Big(\int_{0}^{+\infty}-\int_{0}^{\sqrt{\beta}N(\frac{1}{\sqrt{2}}-\alpha_{N})}-\int_{\sqrt{\beta}N(\frac{1}{\sqrt{2}}+\alpha_{N})}^{+\infty}\Big)\,
e^{-\frac{r^{2}}{2}}r^{N_{\beta}-2}dr\nonumber\\
&=\frac{1}{C_{N\beta}}\int_{0}^{+\infty}e^{-\frac{r^{2}}{2}}r^{N_{\beta}-2}dr+O(e^{-\beta
N^{2(1-\theta)}(1+o(1))})\nonumber\\
&=\frac{\Gamma(\frac{N_{\beta}-1}{2})}{\sqrt{2}\Gamma(\frac{N_{\beta}}{2})}+O(e^{-\beta
N^{2(1-\theta)}(1+o(1))})=\frac{\sqrt{2}}{N\sqrt{\beta}}+O(\frac{1}{N^{2}}).
\end{align}
It is worth emphasizing that the range of $\theta$ in Eqs. (\ref{1})
and (\ref{2}) plays a vital role in the proof of Theorem 2.

Hence Eq.(\ref{zhong zhi ding li hou }) can be reduced to
\begin{equation}
\rho_{\beta HE_{N}}(x_{1})=\rho_{\beta
FTE_{N},1}(\frac{x_{1}}{\xi_{N}(x_{1})})(\frac{\sqrt{2}}{N\sqrt{\beta}}+O(\frac{1}{N^{2}}))
+\Xi_{N}.\label{3.22}
\end{equation}
If we make the change of variables $x_{1}=\sqrt{2N\beta}x$, the
relation (\ref{relation}) implies that
\begin{align}
&\rho_{\beta HE_{N}}(\sqrt{2N\beta}x)=\rho_{\beta
FTE_{N},1}\left(\frac{\sqrt{2N\beta}x}{\xi_{N}(x)}\right)\left(\frac{\sqrt{2}}{N\sqrt{\beta}}+O(\frac{1}{N^{2}})\right)+\Xi_{N}.\nonumber\\
&=\rho_{\beta
FTE_{N}}\left(\sqrt{\frac{N(N-1)}{2}}\frac{\sqrt{2N\beta}x}{\xi_{N}(x)}\right)\left(\frac{\sqrt{2}}{N\sqrt{\beta}}+O(\frac{1}{N^{2}})\right)\sqrt{\frac{N(N-1)}{2}}+\Xi_{N}\nonumber\\
&=\big(\frac{1}{\sqrt{\beta}}+O(\frac{1}{N})\big)\,\rho_{\beta
FTE_{N}}\big(\sqrt{2N}x(1+O(\alpha_{N}))\big)+ \Xi_{N}.
\end{align}
Here we make use of the fact that for the large $N$,
\begin{align}
&\sqrt{\frac{N(N-1)}{2}}\frac{\sqrt{2N\beta}x}{\xi_{N}(x)}=(1+O(\alpha_{N}))\sqrt{2N}x,\nonumber\\
&\left(\frac{\sqrt{2}}{N\sqrt{\beta}}+O(\frac{1}{N^{2}})\right)\sqrt{\frac{N(N-1)}{2}}=\frac{1}{\sqrt{\beta}}+O(\frac{1}{N}).\nonumber
\end{align}
Note that when $0< \theta <0.5$ $\Xi_{N}$ can be rewritten  by
$$\Xi_{N}=O(e^{-\beta
N^{2(1-\theta)}(1+o(1))}),$$ thus we conclude the lemma.
\end{proof}

 Now let us turn to the proof of Theorem 1.
\begin{proof}[Proof of Theorem 1]
Let $f(x)\in C_{c}(\mathbb{R})$. Since $f$ is bounded, using Lemma
\ref{lemma:Globa almost equivalent}, for fixed $0<\theta <0.5$ we
have
\begin{align*}
&\int_{\mathbb{R}}f(x)\sqrt{2N\beta}\rho_{\beta
HE_{N}}(\sqrt{2N\beta}x)dx \nonumber\\
&=\int_{\mathbb{R}}f(x)\sqrt{2N\beta}(\frac{1}{\sqrt{\beta}}+O(\frac{1}{N}))\rho_{\beta
FTE_{N}}(\sqrt{2N}x(1+O(N^{-\theta})))dx+O(e^{-\beta
N^{2(1-\theta)}(1+o(1))})\nonumber\\
&=(1+O(N^{-\theta}))\int_{\mathbb{R}}f(y(1+O(N^{-\theta})))\sqrt{2N}\rho_{\beta
FTE_{N}}(\sqrt{2N}\,y)d\,y+O(e^{-\beta N^{2(1-\theta)}(1+o(1))}).
\end{align*}
The function $f(x)\in C_{c}(\mathbb{R})$ means that for any
$\epsilon>0$, there exists some $\delta(\epsilon)>0$ such that
$|f(x)-f(y)|<\epsilon$ if $|x-y|<\delta$. Hence, there exists
$N_{1}$ such that for any $y\in supp(f)$,
$|y(1+O(N^{-\theta}))-y|<\delta$ for $N>N_{1}$, then
$|f(y(1+O(N^{-\theta})))-1)-f(y)|<\epsilon.$ This shows that
\begin{align*}
&\Big|(1+O(N^{-\theta}))\int_{\mathbb{R}}[f(y(1+O(N^{-\theta})))-f(y)]\sqrt{2N}\rho_{\beta
FTE_{N}}(\sqrt{2N}y)dy\Big|\nonumber\\
&\leq 2\epsilon\int_{\mathbb{R}}\sqrt{2N}\rho_{\beta
FTE_{N}}(\sqrt{2N}y)dy=2\epsilon.\nonumber
\end{align*}
Therefore,
\begin{align*}
&\int_{\mathbb{R}}f(x)\sqrt{2N\beta}\rho_{\beta
HE_{N}}(\sqrt{2N\beta}x)dx \nonumber\\
&=(1+O(N^{-\theta}))\int_{\mathbb{R}}f(y)\sqrt{2N}\rho_{\beta
FTE_{N}}(\sqrt{2N}\,y)d\,y+O(e^{-\beta N^{2(1-\theta)}(1+o(1))})+2
\epsilon.
\end{align*}
Since by (1.5)
$$ \lim_{N\rightarrow\infty}\sqrt{2\beta
N}\large{\rho}_{\beta HE_{N}}(\sqrt{2\beta
N}x)=\rho_{\mathrm{W}}(x),$$ this completes this proof.
\end{proof}

\section{Proof of Theorem \ref{theorem edge}}
The classic result \cite{F1} claims that the order of scaling at the
spectrum edge is $O(N^{-2/3})$. It follows from Lemma
\ref{lemma:Globa almost equivalent} that if
$\alpha_{N}=N^{-\theta}$, $2/3<\theta<1$, then
\begin{equation}
(1+\frac{x}{2N^{2/3}})(1+O(\alpha_{N}))=1+\frac{x}{2N^{2/3}}+O(N^{-\theta}).\nonumber
\end{equation}
The term $O(N^{-\theta})$, comparing with $O(N^{-2/3})$, is a small
perturbation. Therefore the edge scaling limit of fixed trace
$\beta$-HE could be expected. But Lemma \ref{lemma:Globa almost
equivalent} is established for $0<\theta<0.5$. The main difficulty
results from Proposition \ref{proposition 1}. In order to avoid it,
the edge scaling limit will be proved in the weak sense.
Note that the asymptotic results (\ref{1}) and (\ref{2}) will be
frequently used for any $2/3<\theta<1$ in the subsequent proof. We
now turn to the proof of Theorem \ref{theorem edge}.

\begin{proof}[Proof of Theorem \ref{theorem edge}]
Following the similar arguments of Theorem \ref{theorem:global}, by
the basic relation (\ref{integral equation}) between two ensembles,
we have
\begin{align}
\label{divide integral}
&\frac{\sqrt{\beta}N^{5/6}}{\sqrt{2}}\int_{\mathbb{R}}f(x)\rho_{\beta
HE_{N}}\left(\sqrt{2N\beta}(1+\frac{x}{2N^{2/3}})\right)\,dx\nonumber\\
&=\frac{1}{C_{N\beta}}\frac{\sqrt{\beta}N^{5/6}}{\sqrt{2}}\int_{\mathbb{R}}f(x)\int_{\sqrt{2\beta
N}(1+\frac{x}{2N^{2/3}})}^{\infty}e^{-r^{2}/2}r^{N_{\beta}-2}\rho_{\beta
FTE,1}\Big(\frac{\sqrt{2N\beta}(1+\frac{x}{2N^{2/3}})}{r}\Big)drdx\nonumber\\
&=\frac{1}{C_{N\beta}}\frac{\sqrt{\beta}N^{5/6}}{\sqrt{2}}\int_{\mathbb{R}}f(x)\Big(\int_{\sqrt{2\beta
N}(1+\frac{x}{2N^{2/3}})}^{\sqrt{\beta}N(1/\sqrt{2}-\alpha_{N})}
+\int_{\sqrt{\beta}N(1/\sqrt{2}-\alpha_{N})}^{\sqrt{\beta}N(1/\sqrt{2}+\alpha_{N})}+\int_{\sqrt{\beta}N(1/\sqrt{2}+\alpha_{N})}^{\infty}
\Big)\nonumber\\ &\times e^{-r^{2}/2}r^{N_{\beta}-2}\rho_{\beta
FTE,1}\Big(\frac{\sqrt{2N\beta}(1+\frac{x}{2N^{2/3}})}{r}\Big)drdx\nonumber\\
&=I_{1}+I_{2}+I_{3}.\end{align} The first step is to estimate
$I_{1}$. Making the change of variables
\begin{equation}
\label{variable substitution} x=2N^{2/3}(r(1+\frac{y}{2N^{2/3}})-1)
\end{equation}and assuming $|f(x)|\leq M$, $I_{1}$ can be dominated by
\begin{multline*}
I_{1}=\frac{1}{C_{N\beta}}\frac{\sqrt{\beta}N^{5/6}}{\sqrt{2}}\int_{\mathbb{R}}\int_{0}^{\sqrt{\beta}N(1/\sqrt{2}-\alpha_{N})}
1_{\sqrt{2\beta N}r(1+y/2N^{2/3}) \leq r\leq
\sqrt{\beta}N(1/\sqrt{2}-\alpha_{N})}(r) \\ \times
f(2N^{2/3}(r(1+\frac{y}{2N^{2/3}})-1))e^{-r^{2}/2}r^{N_{\beta}-1}\rho_{\beta
FTE,1}\Big(\sqrt{2N\beta}(1+\frac{y}{2N^{2/3}})\Big)dr\,dy
\end{multline*}\begin{align}
\label{I1: result} &\leq M
\frac{1}{C_{N\beta}}\frac{\sqrt{\beta}N^{5/6}}{\sqrt{2}}\int_{\mathbb{R}}\int_{0}^{\sqrt{\beta}N(1/\sqrt{2}-\alpha_{N})}
e^{-r^{2}/2}r^{N_{\beta}-1}\rho_{\beta
FTE,1}\Big(\sqrt{2N\beta}(1+\frac{y}{2N^{2/3}})\Big)drdy\nonumber\\
&=M\frac{1}{C_{N\beta}}\int_{0}^{\sqrt{\beta}N(1/\sqrt{2}-\alpha_{N})}e^{-r^{2}/2}r^{N_{\beta}-1}dr
\frac{\sqrt{\beta}N^{5/6}}{\sqrt{2}}\int_{\mathbb{R}}\rho_{\beta
FTE,1}\Big(\sqrt{2N\beta}(1+\frac{y}{2N^{2/3}})\Big)dy\nonumber\\
&=M\,O(e^{-\beta
N^{2(1-\theta)}(1+o(1))})\frac{\sqrt{\beta}N^{5/6}}{\sqrt{2}}\,\frac{\sqrt{2}N^{1/6}}{\sqrt{\beta}}\int_{\mathbb{R}}\rho_{\beta
FTE,1}(t)dt\nonumber\\
&=O(N\, e^{-\beta N^{2(1-\theta)}(1+o(1))}).
\end{align}
Here we apply (\ref{1}) and $\int_{\mathbb{R}}\rho_{\beta
FTE,1}(y)dy=1$ to obtain the above result. Similarly, making the
same variable substitution as $I_{1}$,  $I_{3}$ can be dominated by
\begin{multline*}
I_{3}=\frac{1}{C_{N\beta}}\frac{\sqrt{\beta}N^{5/6}}{\sqrt{2}}\int_{\mathbb{R}}\int_{\sqrt{\beta}N(1/\sqrt{2}+\alpha_{N})}^{\infty}
f(2N^{2/3}(r(1+\frac{y}{2N^{2/3}})-1))\\ \times
e^{-r^{2}/2}r^{N_{\beta}-1}\rho_{\beta
FTE,1}\Big(\sqrt{2N\beta}(1+\frac{y}{2N^{2/3}})\Big)dr\,dx
\end{multline*}\begin{align}
\label{I3:result} &\leq M
\frac{1}{C_{N\beta}}\frac{\sqrt{\beta}N^{5/6}}{\sqrt{2}}\int_{\mathbb{R}}\int_{\sqrt{\beta}N(1/\sqrt{2}+\alpha_{N})}^{\infty}
e^{-r^{2}/2}r^{N_{\beta}-1}\rho_{\beta
FTE,1}\Big(\sqrt{2N\beta}(1+\frac{y}{2N^{2/3}})\Big)drdx\nonumber\\
&=M\frac{1}{C_{N\beta}}\int_{\sqrt{\beta}N(1/\sqrt{2}+\alpha_{N})}^{\infty}e^{-r^{2}/2}r^{N_{\beta}-1}dr
\frac{\sqrt{\beta}N^{5/6}}{\sqrt{2}}\int_{\mathbb{R}}\rho_{\beta
FTE,1}\Big(\sqrt{2N\beta}(1+\frac{y}{2N^{2/3}})\Big)dy\nonumber\\
&=O(N\,e^{-\beta N^{2(1-\theta)}(1+o(1))}).
\end{align}
The key step is to deal with $I_{2}$. Applying (\ref{relation}) to
$I_{2}$, we find
\begin{align}
I_{2}=\frac{1}{C_{N\beta}\,a_{N\beta}}\frac{N^{5/6}}
{\sqrt{2}}\int_{\mathbb{R}}\int_{\sqrt{\beta}N(\frac{1}{\sqrt{2}}-\alpha_{N})}^{\sqrt{\beta}N(\frac{1}{\sqrt{2}}+\alpha_{N})}
f(x)e^{-r^{2}/2}r^{N_{\beta}-2}\rho_{\beta
FTE_{N}}\Big(\frac{\sqrt{2N}(1+\frac{x}{2N^{2/3}})}{a_{N\beta}r}\Big)drdx
\end{align}
where
\begin{equation}
\label{anbeta} a_{N\beta}=\sqrt{\frac{2}{\beta N(N-1)}}\,.
\end{equation}

Making the change of variables
$x=2N^{2/3}(a_{N\beta}r(1+y/(2N^{2/3}))-1)$,
\begin{multline}
\label{I2: bian xian hou de}
I_{2}=\frac{1}{C_{N\beta}}\frac{N^{5/6}}
{\sqrt{2}}\int_{\mathbb{R}}\int_{\sqrt{\beta}N(\frac{1}{\sqrt{2}}-\alpha_{N})}^{\sqrt{\beta}N(\frac{1}{\sqrt{2}}+\alpha_{N})}
f(2N^{2/3}(a_{N\beta}r(1+\frac{y}{2N^{2/3}})-1))
\\e^{-r^{2}/2}r^{N_{\beta}-1}\rho_{\beta
FTE_{N}}\Big(\sqrt{2N}(1+\frac{y}{2N^{2/3}})\Big)dr\,dy.
\end{multline}
Basing on (\ref{media integral}), it is not difficult to see that
for large $N$,
\begin{align}
\label{result:edge} &\frac{1}{C_{N\beta}}\frac{N^{5/6}}
{\sqrt{2}}\int_{\mathbb{R}}\int_{\sqrt{\beta}N(\frac{1}{\sqrt{2}}-\alpha_{N})}^{\sqrt{\beta}N(\frac{1}{\sqrt{2}}+\alpha_{N})}
f(y)e^{-r^{2}/2}r^{N_{\beta}-1}\rho_{\beta
FTE_{N}}\Big(\sqrt{2N}(1+\frac{y}{2N^{2/3}})\Big)drdy\\
&=\frac{1}{C_{N\beta}}\int_{\sqrt{\beta}N(\frac{1}{\sqrt{2}}-\alpha_{N})}^{\sqrt{\beta}N(\frac{1}{\sqrt{2}}+\alpha_{N})}
e^{-r^{2}/2}r^{N_{\beta}-1}dr\,
\int_{\mathbb{R}}f(y)\frac{N^{5/6}}{\sqrt{2}}\rho_{\beta
FTE_{N}}\left(\sqrt{2N}(1+\frac{y}{2N^{2/3}})\right)dy\nonumber\\
&=(1+O(\frac{1}{N}))\int_{\mathbb{R}}f(y)\frac{N^{5/6}}{\sqrt{2}}\rho_{\beta
FTE_{N}}\left(\sqrt{2N}(1+\frac{y}{2N^{2/3}})\right)dy\nonumber
\end{align}
Next, we will prove that the limit of the difference between $I_{2}$
 and (\ref{result:edge}) is zero as $N$ goes to infinity. Note that
$f(x)\in C_{c}(\mathbb{R})$. For any $\epsilon>0$, there exists some
$\delta(\epsilon)>0$ such that if $|x-y|<\delta$, then
\begin{equation}
|f(x)-f(y)|<\epsilon. \nonumber
\end{equation}
Since $$\sqrt{\beta}N(\frac{1}{\sqrt{2}}-\alpha_{N})\leq r \leq
\sqrt{\beta}N(\frac{1}{\sqrt{2}}+\alpha_{N}),$$ one has $a_{N\beta}r
=1+O(\alpha_{N})$.

 Hence, for any $y\in supp(f)$, there exists
$N_{1}$ not depending on $y$ such that
$$|2N^{2/3}(a_{N\beta}r(1+\frac{y}{2N^{2/3}})-1)-y|=|O(N^{\frac{2}{3}-\theta})-y\,O(N^{-\theta})|<\delta$$
for $N>N_{1}$. Thus,
$$|f\big(2N^{2/3}(a_{N\beta}r(1+\frac{y}{2N^{2/3}})-1)\big)-f(y)|<\epsilon.$$
Pick an interval $[-K,K]$ such that
\begin{equation}
supp(f)\subset[-K,K],\,\{2N^{2/3}(a_{N\beta}r(1+\frac{y}{2N^{2/3}})-1)|\,y\in
supp(f)\}\subset[-K,K].\nonumber
\end{equation}
It follows from Lemma \ref{mainlemma} below that the difference
between
 (\ref{I2: bian xian hou de}) and
(\ref{result:edge}) can be dominated by
\begin{align}
\label{difference} &|I_{2}-(\ref{result:edge})|\leq
\frac{\epsilon}{C_{N\beta}}\frac{N^{5/6}}
{\sqrt{2}}\int_{-K}^{K}\int_{\sqrt{\beta}N(\frac{1}{\sqrt{2}}-\alpha_{N})}^{\sqrt{\beta}N(\frac{1}{\sqrt{2}}+\alpha_{N})}
\,e^{-r^{2}/2}r^{N_{\beta}-1}\rho_{\beta
FTE_{N}}\Big(\sqrt{2N}(1+\frac{y}{2N^{2/3}})\Big)drdy\nonumber\\
&=\epsilon\, \frac{N^{5/6}}{\sqrt{2}}\int_{-K}^{K}\rho_{\beta
FTE_{N}}\left(\sqrt{2N}(1+\frac{y}{2N^{2/3}})\right)dy\,\frac{1}{C_{N\beta}}
\int_{\sqrt{\beta}N(\frac{1}{\sqrt{2}}-\alpha_{N})}^{\sqrt{\beta}N(\frac{1}{\sqrt{2}}+\alpha_{N})}
e^{-r^{2}/2}r^{N_{\beta}-1}dr\nonumber\\  &\leq \epsilon
(1+O(\frac{1}{N}))\frac{N^{5/6}}{\sqrt{2}}\int_{-K}^{K}\rho_{\beta
FTE_{N}}\left(\sqrt{2N}(1+\frac{y}{2N^{2/3}})\right)dy\\
&\leq 2\epsilon M_{K}\nonumber
\end{align}
where the constant $M_{K}$ only depends on $K$. It is worth
mentioning that  we have applied the followng Lemma \ref{mainlemma}
to (\ref{difference}). With the requirement of the continuity of our
argument, Lemma \ref{mainlemma} will be proved after completing this
proof. Now we have proved that
$$\lim_{N\rightarrow\infty}|I_{2}-(\ref{result:edge})|=0.$$
Combining (\ref{divide integral}), (\ref{I1: result}),
(\ref{I3:result}), (\ref{I2: bian xian hou de}) and
(\ref{result:edge}),
\begin{multline*}
\int_{\mathbb{R}}f(x)\frac{\sqrt{\beta}N^{5/6}}{\sqrt{2}}\rho_{\beta
HE_{N}}\left(\sqrt{2N\beta}(1+\frac{x}{2N^{2/3}})\right)\,dx\\
=(1+O(\frac{1}{N}))\int_{\mathbb{R}}f(y)\frac{N^{5/6}}{\sqrt{2}}\rho_{\beta
FTE_{N}}\left(\sqrt{2N}(1+\frac{y}{2N^{2/3}})\right)dy +o(1).
\end{multline*}

It immediately follows from the assumption  of Theorem $\ref{theorem
edge}$ that
\begin{multline*}
\lim_{N\rightarrow\infty}\int_{\mathbb{R}}f(x)\frac{N^{5/6}}{\sqrt{2}}\rho_{\beta
FTE_{N}}\left(\sqrt{2N}(1+\frac{x}{2N^{2/3}})\right)dx
\\=\lim_{N\rightarrow\infty}\int_{\mathbb{R}}f(x)\frac{\sqrt{\beta}N^{5/6}}{\sqrt{2}}\rho_{\beta
HE_{N}}\left(\sqrt{2N\beta}(1+\frac{x}{2N^{2/3}})\right)\,dx.
\end{multline*}
This completes the proof.
\end{proof}
The following lemma will be proved by using a similar argument from
Lemma 4 in \cite{GG1}.
\begin{lemma}
\label{mainlemma} If $\forall\, h(x)\in C_{c}(\mathbb{R})$,
\begin{align}
\label{lemma:at the edge, condition}
\lim_{N\rightarrow\infty}\int_{\mathbb{R}}h(x)\frac{\sqrt{\beta}N^{5/6}}{\sqrt{2}}\rho_{\beta
HE_{N}}\left(\sqrt{2N\beta}(1+\frac{x}{2N^{2/3}})\right)\,dx
\end{align}
exists, then for any fixed $R$,
\begin{equation}
\label{upper bound:at the edge}
\frac{N^{5/6}}{\sqrt{2}}\int_{-R}^{R}\rho_{\beta
FTE_{N}}\left(\sqrt{2N}(1+\frac{y}{2N^{2/3}})\right)dy\leq M_{R}.
\end{equation}
where $M_{R}$ is a constant only depending on $R$.
\end{lemma}
\begin{proof}
If $\sqrt{\beta}N(\frac{1}{\sqrt{2}}-\alpha_{N})\leq u\leq
\sqrt{\beta}N(\frac{1}{\sqrt{2}}+\alpha_{N})$, then there exists
$N_{0}$ such that for $N>N_{0}$,
\begin{equation}
|2N^{2/3}(a_{N\beta}u(1+\frac{y}{2N^{2/3}})-1)|\leq R+1
\end{equation}
where $a_{N\beta}$ is given by (\ref{anbeta}). Let $\eta\in(0,1)$ be
a real number and let $\phi(t)$ be a smooth decreasing function on
$[0,R+1)$ such that $\phi(t) = 1$ for $t \in[0,R+1)$ and $\phi(t)=0$
for $t\geq(1+\eta)(R+1)$. Therefore, we have
\begin{align}
&\frac{N^{5/6}}{\sqrt{2}}\int_{-R}^{R}\rho_{\beta
FTE_{N}}\left(\sqrt{2N}(1+\frac{y}{2N^{2/3}})\right)dy\nonumber\\&\leq
\frac{N^{5/6}}{\sqrt{2}}\int_{\mathbb{R}}\phi(2N^{2/3}(a_{N\beta}u(1+\frac{y}{2N^{2/3}})-1))\rho_{\beta
FTE_{N}}\left(\sqrt{2N}(1+\frac{y}{2N^{2/3}})\right)dy.
\end{align}
Multiplying both sides by $$\frac{1}{C_{N\beta}}
e^{-u^{2}/2}u^{N_{\beta}-1}$$ then integrating  about $u$ on
$$[\sqrt{\beta}N(\frac{1}{\sqrt{2}}-\alpha_{N}),\sqrt{\beta}N(\frac{1}{\sqrt{2}}+\alpha_{N})],$$
one obtains
\begin{align}
\frac{1}{C_{N\beta}}\int_{\sqrt{\beta}N(\frac{1}{\sqrt{2}}-\alpha_{N})}^{\sqrt{\beta}N(\frac{1}{\sqrt{2}}+\alpha_{N})}
e^{-u^{2}/2}u^{N_{\beta}-1}du\,\frac{N^{5/6}}{\sqrt{2}}\int_{-R}^{R}\rho_{\beta
FTE_{N}}\left(\sqrt{2N}(1+\frac{y}{2N^{2/3}})\right)\nonumber
\end{align}
\begin{multline}
\leq\frac{1}{C_{N\beta}}\frac{N^{5/6}}
{\sqrt{2}}\int_{\mathbb{R}}\int_{\sqrt{\beta}N(\frac{1}{\sqrt{2}}-\alpha_{N})}^{\sqrt{\beta}N(\frac{1}{\sqrt{2}}+\alpha_{N})}\nonumber
\phi(2N^{2/3}(a_{N\beta}u(1+\frac{y}{2N^{2/3}})-1))
\\ \times e^{-u^{2}/2}u^{N_{\beta}-1}\rho_{\beta
FTE_{N}}\Big(\sqrt{2N}(1+\frac{y}{2N^{2/3}})\Big)du\,dy.
\end{multline}
On the one hand, combining (\ref{media integral}), it is easy to
observe that the left hand side of the above inequality equals
\begin{equation}
(1+O(\frac{1}{N}))\frac{N^{5/6}}{\sqrt{2}}\int_{-R}^{R}\rho_{\beta
FTE_{N}}\left(\sqrt{2N}(1+\frac{y}{2N^{2/3}})\right) d\,y.
\end{equation}
On the other hand, it follows from (\ref{divide integral}),
(\ref{I1: result}), (\ref{I3:result}) and (\ref{I2: bian xian hou
de}) that the right hand side of the above inequality equals
\begin{equation}
\int_{\mathbb{R}}\phi(y)\frac{\sqrt{\beta}N^{5/6}}{\sqrt{2}}\rho_{\beta
HE_{N}}\left(\sqrt{2N\beta}(1+\frac{y}{2N^{2/3}})\right)\,dy+O(N\,e^{-\beta
N^{2(1-\theta)}(1+o(1))}).
\end{equation}
The existence of the limit of (\ref{lemma:at the edge, condition})
proves this lemma.
\end{proof}

%

\section*{Acknowledgments}
The work of Da-Sheng Zhou and Tao Qian are supported by research
grant of the University of Macau No. FDCT014/2008/A1. Dang-Zheng Liu
thanks Zheng-Dong Wang for his encouragement and support.

\section*{Appendix: equivalence of moments} This appendix presents
their moment equivalence between fixed trace $\beta$-Hermite
ensembles and $\beta$-Hermite ensembles in the large $N$.

Recall that Dumitriu and Edelman's $\beta$-Hermite tri-diagonal
matrix models
 $$
H_{\beta} = \begin{bmatrix}
 a_{_{N}} & b_{_{N-1}}  &       &      \\
 b_{_{N-1}} & a_{_{N-1}}  &\ddots &      \\
     &\ddots&\ddots & {\!b_{1}\!}     \\
     &      &{\!b_{1}\!}&  a_1
\end{bmatrix}
$$
where $a_{j}\sim N(0,1),\ j=1,\cdots, N$ and $\sqrt{2}b_{j}\sim
\chi_{j\beta},\ j=1,\cdots, N-1$. Then proceeding from the analogy
of a fixed energy in classical statistical mechanics, one can define
a `fixed trace' ensemble by the requirement that the trace of
$H_{\beta}^{2}$ be fixed to a number $r^{2}$ with no other
constraint. The number $r$ is called the strength of the ensemble.
The joint probability density function for the matrix elements of
$H_{\beta}$ is therefore given by

\[P_{r}(H_{\beta})=K_{r}^{-1}\,
\delta\big(\frac{1}{r^{2}}\text{tr}\,H_{\beta}^{2}-1\big)
\prod_{j=1}^{N-1}b_{j}^{j\beta-1}\] with
\[K_{r}=\int \cdots \int \delta\big(\frac{1}{r^{2}}\text{tr}\,H_{\beta}^{2}-1\big)
\prod_{j=1}^{N-1}b_{j}^{j\beta-1}d\,a\, d\,b\] where
$d\,a=\prod_{j=1}^{N}d\,a_{j}$ and $d\,b=\prod_{j=1}^{N-1}d\,b_{j}.$

Note that for the $\beta$-Hermite ensemble the joint probability
function
\[P(H_{\beta})=K^{-1}\,
\text{exp}\big(-\frac{1}{2}\text{tr}\,H_{\beta}^{2}\big)
\prod_{j=1}^{N-1}b_{j}^{j\beta-1}\] with
\[K=\int \cdots \int \text{exp}\big(-\frac{1}{2}\text{tr}\,H_{\beta}^{2}\big)
\prod_{j=1}^{N-1}b_{j}^{j\beta-1} d\,a\, d\,b.\] When we choose the
number $r^{2}$ as the average of $H_{\beta}^{2}$, i.e.,
\[\langle\text{tr}\,H_{\beta}^{2}\rangle=
K^{-1}\int \cdots \int
\text{tr}\,H_{\beta}^{2}\,\text{exp}\big(-\frac{1}{2}\text{tr}\,H_{\beta}^{2}\big)
\prod_{j=1}^{N-1}b_{j}^{j\beta-1} d\,a\, d\,b=r^{2},\] then for any
fixed value of the sum
\[s=\sum_{j=1}^{N}\eta_{j}^{(a)}+\sum_{j=1}^{N-1}\eta_{j}^{(b)},\ \ \ \eta_{j}^{(a)},\ \eta_{j}^{(b)} \geq 0,\]
the ratio the moments
$$M_{r}(N,\eta)=\Big\langle\prod_{j=1}^{N}(a_{j})^{\eta_{j}^{(a)}}\prod_{j=1}^{N-1}(b_{j})^{\eta_{j}^{(b)}}\Big\rangle_{r}$$
and
$$M(N,\eta)=\Big\langle\prod_{j=1}^{N}(a_{j})^{\eta_{j}^{(a)}}\prod_{j=1}^{N-1}(b_{j})^{\eta_{j}^{(b)}}\Big\rangle$$
tends to unity as the number of dimensions $N$ tends to infinity.
The subscript $r$ and non-subscript denote that the average is taken
in the fixed trace and $\beta$-Hermite ensembles, respectively.

First, let's calculate
$r^{2}=\langle\text{tr}\,H_{\beta}^{2}\rangle$ with a basic
manipulation in statistical mechanics. Write
\[g_{\beta}(\lambda)=c_{_{H}}^{\beta}|\Delta(\lambda)|^{\beta}\text{exp}\big(-t\sum_{j=1}^{N}\lambda_{j}^{2}\big)\]
where
\[c_{_{H}}^{\beta}=(2t)^{N/2+\beta N(N-1)/4}\,(2\pi)^{-N/2}\prod_{j=1}^{N}\frac{\Gamma (1+\beta/2)}{\Gamma(1+j\beta/2)}.\]
Note that $g_{\beta}(\lambda)$ with $t=1/2$ corresponds to the joint
probability distribution of eigenvalues for $\beta$-Hermite
ensembles \cite{DE1}. A partial differentiation with respect to $t$
and setting $t=1/2$ gives
\[\langle\text{tr}\,H_{\beta}^{2}\rangle=2(N/2+\beta N(N-1)/4).\]

Next, to calculate $M_{r}(N,\eta)$, substitute
$(2\xi)^{-1/2}\,r\,a_{j}$ for $a_{j}$ and $(2\xi)^{-1/2}\,r\,b_{j}$
for $b_{j}$ where $\xi$ is a parameter. This gives
\begin{align*}
M_{r}&(N,\eta)\big(\frac{2\xi}{r^{2}}\big)^{N/2+\beta N(N-1)/4+s/2}\\
 &=K_{r}^{-1}\int \cdots \int
\delta\big(\frac{1}{2\xi}\text{tr}\,H_{\beta}^{2}-1\big)
\prod_{j=1}^{N}(a_{j})^{\eta_{j}^{(a)}}\prod_{j=1}^{N-1}(b_{j})^{\eta_{j}^{(b)}}
\prod_{j=1}^{N-1}b_{j}^{j\beta-1}d\,a\, d\,b.
\end{align*}
Multiplying both sides by $e^{-\xi}$ and integrating on $\xi$ from 0
to $\infty $, we get
\begin{align*}
M_{r}&(N,\eta)\Gamma(L+s/2+1)\,L^{-L-s/2}\\
 &=K_{r}^{-1}\int \cdots \int
\text{exp}\big(-\frac{1}{2}\text{tr}\,H_{\beta}^{2}\big)
\prod_{j=1}^{N}(a_{j})^{\eta_{j}^{(a)}}\prod_{j=1}^{N-1}(b_{j})^{\eta_{j}^{(b)}}
\prod_{j=1}^{N-1}b_{j}^{j\beta-1}d\,a\, d\,b
\end{align*}
where we have put\[L=\frac{1}{2} r^{2}=N/2+\beta N(N-1)/4.\] or
\[M_{r}(N,\eta)=\frac{L^{L+s/2}}{\Gamma(L+s/2+1)}\frac{K}{K_{r}}
M(N,\eta).\] Setting $\eta_{j}^{(a)}=\eta_{j}^{(b)}=0$ in the above
and using the normalization condition $M_{r}(N,0)=M(N,0)=1$, we get
the ratio of the constants $K$ and $K_{r}$. Substituting this ratio
we then obtain
\[M_{r}(N,\eta)=\frac{L^{s/2}\Gamma(L+1)}{\Gamma(L+s/2+1)}
M(N,\eta).\] As $N\rightarrow \infty, L\rightarrow \infty$, and we
can use Stirling's formula for the gamma function for the large $x$
\[\Gamma(x+1)=x^{-x}\,e^{-x}\sqrt{2\pi x}\,[1+O(1/x)],
\]
to prove the asymptotic equality of all the finite moments $s\ll N$.

In sum, the result of moment equivalence can be stated as follows:
\begin{theorem}With the above notation $M_{r}(N,\eta)$and $M(N,\eta)$, we have
\[\lim_{N\rightarrow \infty}\frac{M_{r}(N,\eta)}{M(N,\eta)}=1.\]
\end{theorem}

\end{document}